\documentclass[reqno,11pt,oneside]{amsart}
\numberwithin{equation}{section}
\usepackage{amsthm,amssymb}
\newtheorem{Lemma}{Lemma}[section]
\newtheorem{Theorem}[Lemma]{Theorem}
\newtheorem{Corollary}[Lemma]{Corollary}
\theoremstyle{definition}
\newtheorem{Remark}{Remark}[section]

\usepackage{geometry}
\usepackage{tensor}
\usepackage{enumerate}
\usepackage[colorlinks=true]{hyperref}
\newcommand{\pa}{\partial}
\newcommand{\Fg}{\mathfrak{g}}
\newcommand{\Fh}{\mathfrak{h}}
\newcommand{\Fl}{\mathfrak{l}}

\newcommand{\Fu}{\mathfrak{u}}
\newcommand{\Fsl}{\mathfrak{sl}}
\newcommand{\al}{\alpha}
\newcommand{\be}{\beta}

\newcommand{\la}{\lambda}

\newcommand{\De}{\Delta}

\newcommand{\op}{\operatorname}
\title[ABRR Summation Formulas for Relative Extremal Projectors]{ABRR Summation Formulas\\for Relative Extremal Projectors}
\author{Jonas T. Hartwig}
\address{Department of Mathematics, Iowa State University, Ames, IA-50011, USA}
\email{jth@iastate.edu}
\urladdr{http://jthartwig.net}
\thanks{The author was supported in part by the Army Research Office grant W911NF-24-1-0058.
The author is grateful to Hjalmar Rosengren for pointing out the reference \cite{Clebsch1862}.}
\date{}

\begin{document}
\begin{abstract}
In 2004, Khoroshkin proved that the extremal projector is equivalent to the dynamical twist. The Arnaudon-Buffenoir-Ragoucy-Roche (ABRR) equation, satisfied by the dynamical twist, yields a recursive formula for the coefficients in the extremal projector.
The resulting summation formula is uniform and powerful in its applications to representation theory and Mickelsson-Zhelobenko reduction algebras.

In this paper we generalize this recursive formula to the relative extremal projector, introduced by Conley and Sepanski in 2003. When the Levi subalgebra is the Cartan subalgebra, we recover the usual ABRR recursion. We also find a compact and explicit expression for the solution to the recursion.
Our setting includes infinite-dimensional contragredient Lie superalgebras, Kac-Moody algebras, basic classical Lie superalgebras, and finite-dimensional reductive Lie algebras.
\end{abstract}
\maketitle 

\section{Introduction}\label{sec:introduction}
\subsection{Historical Background and Motivation}
The extremal projector is a unique object in mathematics. It seems to have been discovered and subsequently rediscovered independently several times. In 1862, Clebsch \cite[Eq.~(3.)]{Clebsch1862} proved that that if $v=v(x,y,z)$ is a smooth homogeneous function of degree $n$, then 
\begin{equation}
	u = v-\frac{(x^2+y^2+z^2)\cdot \De(v)}{2(2n-1)}+\frac{(x^2+y^2+z^2)^2\cdot\De^2(v)}{2\cdot 4\cdot (2n-1)\cdot (2n-3)}-\cdots
\end{equation}
solves the Laplace equation $\De u=0$. Introducing the operators
\begin{equation}
	e = \frac{1}{2}\sum_i (\pa_i)^2,\qquad f = -\frac{1}{2}\sum_i (x_i)^2,\qquad h = -\frac{1}{2}\sum_i (x_i\pa_i+\pa_ix_i),
\end{equation}
Clebsch's operator $v\mapsto u$ may be expressed as
\begin{equation}\label{eq:sl2-projector}
	P = 1 - \frac{1}{h+2}fe + \frac{1}{2!}\frac{1}{(h+2)(h+3)}f^2e^2 - \frac{1}{3!}\frac{1}{(h+2)(h+3)(h+4)}f^3e^3+\cdots.
\end{equation}
In a preprint published 98 years later, physicist Per-Olov Löwdin writes down the same expression \cite[Eq. (31)]{Low1958} (later published in \cite{Low1964}) in the context of describing angular momentum wave functions for composite systems in quantum mechanics.
This operator was further studied as a formal series by Shapiro \cite{Shapiro1965}.

Today, we recognize the expression \eqref{eq:sl2-projector} as the \emph{extremal projector} \cite{AST1971,Zh1989,Tolstoy2005,Tolstoy2011} for the three-dimensional simple Lie algebra.
It can be defined as an element of a completion $\bar U$ of a localization of the universal enveloping algebra of $\Fsl(2,\mathbb{C})$.

The remarkable properties of $P$, that in fact define the series \eqref{eq:sl2-projector} uniquely, are:
\begin{equation}
	eP=0,\qquad Pf=0,\qquad hP=Ph,\qquad P^2=P,\qquad
	P\equiv 1 \mod \bar Ue.
\end{equation}
This means that $P$ is a projector onto the space of extremal vectors
\begin{equation}
	V^+ =\{v\in V\mid e\cdot v=0\},
\end{equation}
for any $\bar U$-module $V$, hence the name ``extremal projector''.
The space of extremal (also known as singular) vectors plays an important role in representation theory. One explanation comes from reciprocity:
\begin{equation}
	\op{Hom}_{\Fsl_2}(M(\la),V)\cong (V^+)_\la
\end{equation}
The space of intertwining operators from a Verma module $M(\la)$ of highest weight $\la$ to a finite-dimensional representation $V$ (or object of category $\mathcal{O}$) is naturally isomorphic to the $\la$-weight space in the extremal space. Simply put: \emph{the extremal projector is a tool for constructing morphisms}.
Phrased this way, it is clear why it has been such a popular tool throughout representation theory and theoretical physics.
For more details on applications, we refer to \cite{Tolstoy2005,Tolstoy2011} and references therein.

\subsection{Higher Rank: Multiplicative Approaches}

In 1971, Asherova-Smirnov-Tolstoy proved the existence of the extremal projector $P_\Fg$ for an arbitrary simple finite-dimensional complex Lie algebra. It has the following properties:
\begin{equation}\label{eq:Pg-properties}
	\Fg_+ P_\Fg=0,\quad P_\Fg\Fg_-=0,\quad [\Fh,P_\Fg]=0,\quad P_\Fg^2=P_\Fg,\quad P_\Fg\equiv 1\mod \bar U\Fg_+.
\end{equation}
Here $\Fg=\Fg_-\oplus\Fh\oplus\Fg_+$ denotes a triangular decomposition. They were able to prove that $P_\Fg$ can be expressed as a product of shifted versions of $P_{\Fsl_2}$, one factor for each positive root of $\Fg$.
This factorization method turned out to be wildly successful and was clarified and extended over several decades to Lie superalgebras, quantum groups, Kac-Moody algebras. We refer to \cite{Tolstoy2011} for a review and summary of these results. For a review and details of the semisimple case, see also \cite{Zh1989}. Further perspectives on the projector were given in \cite{KhoOgi2008} and \cite{HW2022}.
Zhelobenko \cite[Eq.~(2.6)]{Zh1993} also provided a different type of factorization, with infinitely many commuting factors, generalizing the $\Fsl_2$ product formulas already appearing in \cite{Low1958,Low1964,Shapiro1965}.

\subsection{Relative Extremal Projectors}
In 2003, Conley and Sepanski \cite{CS03} introduced the notion of a \emph{relative extremal projector} $P_{\Fg,\Fl}$ associated to a pair $(\Fg,\Fl)$ where $\Fg$ is a complex finite-dimensional reductive Lie algebra and $\Fl$ is the Levi subalgebra of a standard parabolic in $\Fg$. The properties \eqref{eq:Pg-properties} are generalized to
\begin{equation}
	\Fu_+ P_{\Fg,\Fl}=0,\quad P_{\Fg,\Fl}\Fu_-=0,\quad [\Fl,P_{\Fg,\Fl}]=0,\quad P_{\Fg,\Fl}^2=P_{\Fg,\Fl},\quad P_{\Fg,\Fl}\equiv 1\mod \bar U\Fu_+.
\end{equation}
Here $\Fg=\Fu_-\oplus \Fl\oplus \Fu_+$.
They proved that it can be expressed as infinite commuting products, generalizing Zhelobenko's result \cite{Zh1993}. Further commutative and non-commutative product formulas for $P_{\Fg,\Fl}$ were given in \cite{CS05,CS15}.

We have avoided discussing the details of the completion of the enveloping algebra to which these extremal projectors belong. The reason is one of pragmatism: Zhelbenko proved, e.g. in \cite[Theorem 1]{Zh1989}, that the completion is isomorphic to the algebra $E(\Fg)=\oplus_{\la\in\Fh^\ast}\op{End}_{\text{Mod-}U'(\Fh)}(M)_\la$ where $M$ is the localized universal Verma module $M=U'(\Fg)/U'(\Fg)\Fg_+$, $U'(\Fg)=S^{-1}U(\Fg)$, $S=U(\Fh)\setminus\{0\}$. From our perspective, $E(\Fg)$ \emph{is} the (Yoneda) completion, see \cite{HW2022,FH2025}.
Thus, in this paper we regard the relative extremal projector (operator) as a linear operator on the localized universal Verma module $M$. The central task is to express this operator using elements of the Lie algebra.

\subsection{Higher Rank: Additive Approaches}

The theory of factorizations of extremal projectors is deep, and filled with surprising discoveries.
However, the actual application of the projector often requires expanding the product and rewriting the terms so that positive root vectors are on the right and negative on the left.
This can be a frustrating and error-prone exercise in practice.
Conley and Sepanski proved the existence of such summation formulas in the relative case, \cite[Thm.~8]{CS03}, but the coefficients are not easily computable; they are expressed in terms dual bases with respect to the Shapovalov form. Computing them means solving an infinite tower of ever larger systems of linear equations. Nevertheless, we utilized their ideas at a crucial moment during the writing of this paper, see Remark \ref{rem:alternative-route-to-solution}.

The astonishing fact is that there is a completely different route to the extremal projector, which is so simple it could easily have been found in the 1970s (or 1870s!). It directly leads to an explicit summation formula for the projector which is ``normal ordered'' from the start, and all coefficients are explicitly known. The method is universal; no special features of root systems are needed.
It only uses the quadratic Casimir, and nothing else. This method is detailed in the present paper, and it is inspired by the ideas in \cite{Khoroshkin2004,ABRR}.

The convoluted route to this discovery goes through integrable face models in statistical mechanics, and its associated dynamical Yang-Baxter equation and dynamical twist.
The main point is that Khoroshkin \cite{Khoroshkin2004} showed that the dynamical twist
from the work of Arnaudon-Buffenoir-Ragoucy-Roche (ABRR) \cite{ABRR}, is just once again the extremal projector in disguise. 
Furthermore, ABRR found an implicit equation for the dynamical twist, which allows one to compute its terms as an infinite series, with explicit coefficients.
We refer to \cite{ABRR,EV1999,Khoroshkin2004} and related references.

The result, found in \cite{Khoroshkin2004}, is that the extremal projector $P_\Fg$ associated to a reductive finite-dimensional complex Lie algebra $\Fg$ can be expressed as
\begin{equation}
	P_\Fg = 1+\sum_{\be\in\De_+}P_\Fg(\be)f_{\be_1}e_{\be_1}+\sum_{\be_1,\be_2\in\De_+} P_\Fg(\be_1,\be_2)f_{\be_1}f_{\be_2}e_{\be_2}e_{\be_1}+\cdots,
\end{equation}
where the coefficients $P_\Fg(\be_1,\ldots,\be_\ell)$ for $\be_i\in Q_{\succ 0}$ are given by the ABRR recursion 
\begin{equation}\label{eq:ABRR-intro}
	P_\Fg(\be_1,\ldots,\be_\ell) = - \psi_{\be_1}P_\Fg(\be_1+\be_2,\be_3,\ldots,\be_\ell),
\end{equation}
with initial condition $P_\Fg(\be)=-\psi_\be$ and
\begin{equation}\label{eq:psi-def-intro}
	\psi_\la = \frac{1}{h_\la + (\rho,\la)+\frac{1}{2}(\la,\la)}.
\end{equation}
Here $(\cdot,\cdot)$ is a non-degenerate invariant symmetric bilinear form on $\Fg$, $e_\be$ and $f_\be$ are dual bases for $\Fg_+$ respectively $\Fg_-$, and $\la\mapsto h_\la$ the inverse of $h\mapsto (h,\cdot)$, and $Q_{\succ 0}=Q_+\setminus\{0\}$ is the nonzero elements in the positive cone in the root lattice, see Section \ref{sec:example-reductive} for details.

Note that the recursion \eqref{eq:ABRR-intro} is easy to solve explicitly:
\begin{equation}\label{eq:ABRR-solution-intro}
	P_\Fg(\be_1,\ldots,\be_\ell)=(-1)^\ell \psi_{\be_1}\psi_{\be_1+\be_2}\cdots\psi_{\be_1+\be_2+\cdots+\be_\ell}.
\end{equation}

\subsection{Main Result}
For simplicity we only state the main result here for the case of finite-dimensional reductive Lie algebras. The more general results of Theorems \ref{thm:relative-extremal-projector-sum}, \ref{thm:relative-ABRR}, \ref{thm:explicit} apply in the generality detailed in Section \ref{sec:setup}.

\begin{Theorem}
Let $\Fg$ be a reductive finite-dimensional complex Lie algebra 
 and let $\Fl$ be the Levi subalgebra of a standard parabolic subalgebra of $\Fg$.
The relative extremal projector can be expanded as follows:
\begin{equation}
P_{\Fg,\Fl}=1+\sum_{\be\in\De_+}P_{\Fg,\Fl}(\be)f_\be e_\be + \sum_{\be_1,\be_2\in\De_+}P_{\Fg,\Fl}(\be_1,\be_2)f_{\be_1}f_{\be_2}e_{\be_2}e_{\be_1}+\cdots 
\end{equation}
where the coefficients $P_{\Fg,\Fl}(\be_1,\ldots,\be_\ell)$ for $\be_i\in Q_{\succ 0}$ are determined recursively by
\begin{equation}\label{eq:relative-ABRR-intro}
P_{\Fg,\Fl}(\be_1,\ldots,\be_\ell) = -\psi_{\be_1}P_{\Fg,\Fl}(\be_1+\be_2,\be_3,\ldots,\be_\ell)+\delta_{\be_1\in Q_+(\Fl)}\psi_{\be_1}P_{\Fg,\Fl}(\be_2,\be_3,\ldots,\be_\ell)^{\be_1}
\end{equation}
and $P_{\Fg,\Fl}(\beta)=(-1+\delta_{\be_1\in Q_+(\Fl)})\psi_{\be_1}$, with $\psi_\be$ defined by \eqref{eq:psi-def-intro}.
Here $\delta_{\be\in Q_+(\Fl)}$ is $1$ if $\be$ is in the Levi cone $Q_+(\Fl)$ and is zero otherwise. 
The superscript $\be_1$ refers to the action of $\Fh^\ast$ on the fraction field of $U(\Fh)$ by shift automorphisms, determined by $h^\be=h+\be(h)$.
If all $\be_i$ belong to $Q_+(\Fl)$, then $P_{\Fg,\Fl}(\be_1,\ldots,\be_\ell)=0$. If at least one $\be_i$ is not in $Q_+(\Fl)$,
letting $m=\min\{i\mid\be_i\notin Q_+(\Fl)\}$, the explicit solution to \eqref{eq:relative-ABRR-intro} can be written
\begin{align}
P_{\Fg,\Fl}(\be_1,\ldots,\be_\ell)=
\sum_{k=0}^{m-1}(-1)^{\ell-k} &\Big(\psi_{\be_1+\cdots+\be_k}\psi_{\be_2+\cdots+\be_k}^{\be_1}\cdots\psi_{\be_k}^{\be_1+\cdots+\be_{k-1}} \nonumber\\
 &\cdot \psi_{\be_{k+1}}^{\be_1+\cdots+\be_k}\psi_{\be_{k+1}+\be_{k+2}}^{\be_1+\cdots+\be_k}\cdots \psi_{\be_{k+1}+\cdots+\be_\ell}^{\be_1+\cdots+\be_k}\Big).
 \label{eq:relative-ABRR-solution-intro}
\end{align}
\end{Theorem}

Note that when $\Fl=\Fh$, the second term in the recursion \eqref{eq:relative-ABRR-intro} is always zero, and we recover the non-relative ABRR recursion \eqref{eq:ABRR-intro}.
Similarly, in this case $m=1$ and \eqref{eq:relative-ABRR-solution-intro} reduces to \eqref{eq:ABRR-solution-intro}.

\subsection{Summary of Contents}
In Section~\ref{sec:setup}, we introduce the assumptions and notation of our setting.
In Section~\ref{sec:resolution_of_the_identity} we use the quadratic Casimir to write down an element $E[\la]$ of the centralizer $U'(\Fg)^\Fh$ which acts as the identity on the weight space $M_{-\la}$ of the universal Verma module $M$.
Then, in Section~\ref{sec:the_resolved_recursion}, we prove what we call the \emph{resolved recursion} for the relative extremal projector coefficients, which is used in Section~\ref{sec:the_relative_abrr_recursion} to derive the ABRR type recursion for these coefficients.
In Section~\ref{sec:an_explicit_formula_for_the_coefficients}, we give an explicit formula for the solution to the relative ABRR recursion.
Lastly, in Section~\ref{sec:examples}, we write out the expressions for the relative extremal projector coefficients up to length $\ell=3$.
We also provide complete details in the case of finite-dimensional reductive Lie algebras, deriving the expression for $\psi_\la^\mu$.
The case of (not necessarily finite-dimensional) contragredient Lie superalgebras is also treated.

\section{Setup} \label{sec:setup}
\subsection{Main Assumptions}
We work over a field $\mathbb{F}$ of characteristic $0$.
Let $\Fg$ be a (not necessarily finite-dimensional) Lie superalgebra, with a vector superspace decomposition into three subalgebras
\begin{equation}
	\Fg=\Fg_-\oplus\Fh\oplus\Fg_+
\end{equation}
 satisfying the following properties:
\begin{enumerate}
\item $\Fh=\Fh_{\bar 0}$ and is
 maximal abelian in $\Fg$;
\item 
 $\Fg_+$ and $\Fg_-$ have bases $\{e^i_\be\}_{\be\in\De_+, i\in I_\be}$, $\{f^i_\be\}_{\be\in\De_+, i\in I_\be}$ of homogeneous weight vectors ($\De_+$ is a subset of $\Fh^\ast$, while $I_\be$ are index sets):
 \begin{equation}\label{eq:chevalley}
 	[h,e_\be^i]=\be(h)e_\be^i,\qquad [h,f_\be^i]=-\be(h)f_\be^i,\qquad\forall h\in\Fh;
 \end{equation}
\item there exists an $\Omega_0\in U(\Fh)$ such that the following formal expression
defines an operator on $U(\Fg)/U(\Fg)\Fg_+$ commuting with the action of $U(\Fg)$:
\begin{equation}\label{eq:casimir}
	\Omega=\Omega_0+\sum_{\be\in\De_+} \sum_{i\in I_\be}f^i_\be e^i_\be.
\end{equation}
In particular, we require that for each $v\in U(\Fg)/U(\Fg)\Fg_+$,
the expression $f_\be^i e_\be^i v$ is nonzero for at most finitely many $(\be,i)$;
\item \label{it:Omega0}
any element in $U(\Fg)$ commuting with $\Omega_0$ also commutes with all of $\Fh$.
\end{enumerate}

\subsection{Further Notation}
\subsubsection{Positive Cone} Let $Q_+=Q_+(\Fg)\subseteq\Fh^\ast$ denote the \emph{positive cone} of the root lattice, defined as
\begin{equation}
	Q_+ = \mathbb{Z}_{\ge 0}\De_+ = \{k_1\be_1+\cdots+k_n\be_n\mid k_i\in\mathbb{Z}_{\ge 0},\, \be_i\in\De_+\}.
\end{equation}
The positive cone $Q_+$, being an additive submonoid of $\Fh^\ast$, provides a partial order $\preceq$ on $\Fh^\ast$ defined by
\begin{equation}
\la\preceq\mu \quad\Longleftrightarrow\quad \mu-\la\in Q_+.
\end{equation}
$\la\prec \mu$ means $\la\preceq\mu$ and we put
\begin{equation}
	Q_{\succ 0}=\{\la\in Q\mid \la\succ 0\}=Q_+\setminus\{0\}.
\end{equation}

\subsubsection{Localization} Let $S=U(\Fh)\setminus\{0\}$. Since $U(\Fg)$ is generated by weight vectors for $\Fh$, $S$ is a left and right Ore denominator set in $U(\Fg)$.
Let $U'(\Fg)=S^{-1}U(\Fg)$ denote the localization. In particular, $U'(\Fg)$ contains the fraction field $U'(\Fh)$ of $U(\Fh)$.
\subsubsection{Shift Automorphisms} For $\la\in \Fh^\ast$, we let $(\,\cdot\,)^\la:U'(\Fh)\to U'(\Fh),\, \phi\mapsto \phi^\la$, denote the unital $\mathbb{F}$-algebra automorphism determined by the property
\begin{equation}\label{eq:shift-def}
	h^\la = h+\la(h),\quad\forall h\in\Fh.
\end{equation}
The following identities in $U'(\Fg)$ (proved by induction on $\ell$ using \eqref{eq:chevalley}) are useful:
\begin{align}\label{eq:shift}
	f_{\be_1}^{i_1}f_{\be_1}^{i_2}\cdots f_{\be_\ell}^{i_\ell} \phi &= \phi^{\be_1+\be_2+\cdots+\be_\ell}f_{\be_1}^{i_1}f_{\be_2}^{i_2}\cdots f_{\be_\ell}^{i_\ell},\quad\forall \ell\in\mathbb{Z}_{\ge 0},\,\be_i\in\De_+,\,\phi\in U'(\Fh), \\ 
	\phi e_{\be_1}^{i_1}e_{\be_1}^{i_2}\cdots e_{\be_\ell}^{i_\ell} &= e_{\be_1}^{i_1}e_{\be_2}^{i_2}\cdots e_{\be_\ell}^{i_\ell}\phi^{\be_1+\be_2+\cdots+\be_\ell},\quad\forall \ell\in\mathbb{Z}_{\ge 0},\,\be_i\in\De_+,\,\phi\in U'(\Fh).
\end{align}
\subsubsection{Universal Verma}
The \emph{(localized) universal Verma module} $M$ is
\begin{equation}
M=U'(\Fg)/I,\quad I=U'(\Fg)\Fg_+.
\end{equation}
It is a $(U'(\Fg),U'(\Fh))$-bimodule.
In particular, $\Fh$ acts via the adjoint action on $M$:
\begin{equation}
	h\cdot m = hm-mh,\qquad \forall h\in\Fh,\, m\in M.
\end{equation}
Let $\boldsymbol{1}=1+I$ denote the generator for $M$ as a left $U'(\Fg)$-module.
We have
\begin{equation}
	M = \bigoplus_{\la\in Q_+} M_{-\la},\qquad M_{-\la}=\{m\in M\mid h\cdot m=-\la(h)m,\forall h\in\Fh\}.
\end{equation}
The minus is due to $M\cong U'(\Fh)U(\Fg_-)$ as $\Fh$-modules with respect to the adjoint action.
\subsubsection{Extremal Projector}
The \emph{extremal projector (operator)}
\begin{equation}
P_\Fg:M\to M	
\end{equation}
is the linear operator on $M$ defined by
\begin{equation}
P_\Fg\big|_{M_{-\la}} =
\begin{cases}
 \op{Id}_{M_{-\la}},&\la=0,\\
0,&\forall \la\in Q_{\succ 0}.
\end{cases}
\end{equation}
\subsubsection{$Q_+$-Grading on the Cyclic Subalgebra}
Let $U'(\Fg)_\la$ denote the weight space (with respect to adjoint action of $\Fh$) in $U'(\Fg)$ of weight $\la\in\Fh^\ast$. In particular, $U'(\Fg)_0$ is the centralizer of $\Fh$ in $U'(\Fg)$, also known as the \emph{cyclic subalgebra}. By the PBW theorem, each element of $U'(\Fg)_0$ is a linear combination of elements of the form $X=\phi F_\la E_\la$ where $\phi\in U'(\Fh), F_\la\in U(\Fg_-)_{-\la}, E_\la\in U(\Fg_+)_\la$. This defines a vector superspace $Q_+$-grading on $U'(\Fg)_0$ in which $X$ has degree $\la$. We use square brackets to refer to this grading:
\begin{equation}
U'(\Fg)_0=\bigoplus_{\la\in Q_+} U'(\Fg)_0[\la],\qquad U'(\Fg)_0[\la] = U'(\Fh)U(\Fg_-)_{-\la} U(\Fg_+)_\la.
\end{equation}
Note that
\begin{equation}\label{eq:support-property}
	U'(\Fg)_0[\la] M_{-\mu} \neq 0 \quad \Longrightarrow \quad \la \preceq \mu.
\end{equation}
Indeed, $U'(\Fg)_0[\la] M_{-\mu} \neq 0 \Rightarrow 0\neq U(\Fg_+)_\la M_{-\mu}\subseteq M_{-\mu+\la}\Rightarrow -\mu+\la\preceq 0$. 

\subsection{Relative Setting}
To introduce the relative setting we additionally choose a subset of $\De_+$, denoted $\De_+(\Fl)$, with the property that whenever $\be_1,\be_2,\ldots,\be_k\in\De_+$ we have
\begin{equation}\label{eq:levi-cone-property}
\be_1+\be_2+\cdots+\be_k \in \mathbb{Z}_{\ge 0}\De_+(\Fl) \;\Longrightarrow\; \be_i\in\De_+(\Fl) \text{ for all $i=1,2,\ldots,k$}.
\end{equation}
Put
\begin{gather*}
	\Fl_+ = \op{span}\{e_\be^i\mid \be\in\De_+(\Fl),\, i\in I_\be\},\quad 
	\Fl_- = \op{span}\{f_\be^i\mid \be\in\De_+(\Fl),\, i\in I_\be\},\\
	\Fl = \Fl_-\oplus\Fh\oplus\Fl_+,\\
	\Fu_+ = \op{span}\{e_\be^i\mid \be\in \De_+\setminus\De_+(\Fl),\, i\in I_\be\},\quad 
	\Fu_- = \op{span}\{f_\be^i\mid \be\in \De_+\setminus\De_+(\Fl),\, i\in I_\be\}.
\end{gather*}
Then $\Fl$ and $\Fu_\pm$ are subalgebras of $\Fg$, and
\[ 
\Fg=\Fu_-\oplus\Fl\oplus\Fu_+,\qquad [\Fl,\Fu_\pm]\subseteq\Fu_\pm.
\] 
We call $\Fl$ the \emph{Levi subalgebra} associated to $\De_+(\Fl)$.
In the extreme case $\De_+(\Fl)=\emptyset$ we have $\Fl=\Fh$, $\Fu_\pm=\Fg_\pm$. In the other extreme case $\De_+(\Fl)=\De_+$ we have $\Fl=\Fg$, $\Fu_\pm=0$.
The \emph{Levi cone} $Q_+(\Fl)$ is defined by
\begin{equation}
	Q_+(\Fl) = \mathbb{Z}_{\ge 0} \De_+(\Fl) = \{k_1\be_1+\cdots+k_n\be_n\mid k_i\in\mathbb{Z}_{\ge 0},\, \be_i\in \De_+(\Fl) \}.
\end{equation}
Following \cite{CS03}, we define the \emph{relative extremal projector (operator)}
\begin{equation}
P_{\Fg,\Fl}:M\to M	
\end{equation}
to be the linear operator on the localized universal Verma module $M=U'(\Fg)/U'(\Fg)\Fg_+$ given by
\begin{equation}\label{eq:relative-extremal-projector-operator-def}
P_{\Fg,\Fl}\big|_{M_{-\la}} =
\begin{cases}
 \op{Id}_{M_{-\la}},&\forall \la\in Q_+(\Fl),\\
0,&\forall \la\in Q_+\setminus Q_+(\Fl).
\end{cases}
\end{equation}
We emphasize that $M$ is the same as in the non-relative case. Note that
\begin{equation}
	P_{\Fg,\Fh} = P_{\Fg},\qquad P_{\Fg,\Fg}=\op{Id}_M.
\end{equation}
Further properties include
\begin{equation}
	(P_{\Fg,\Fl})^2=P_{\Fg,\Fl},\qquad P_\Fg=P_{\Fg,\Fl}P_{\Fl}=P_{\Fl}P_{\Fg,\Fl},\qquad [U(\Fl),P_{\Fg,\Fl}]\Big|_M=0.
\end{equation}
Here and throughout we use the notation $\Big|_M$ to mean ``as operators on $M$''.

\section{Resolution of the Identity} \label{sec:resolution_of_the_identity}

We derive an explicit element $E[\la]\in U'(\Fg)_0[\la]$ which acts as the identity operator on the weight space $M_{-\la}$ of the localized universal Verma module $M$. Following \cite{ForQuiSki2025}, we call such an expansion the \emph{resolution of the identity (on $M_{-\la}$)}. This will be the crucial tool in deriving what we call the \emph{resolved recursion} for the relative extremal projector.

By Assumption \eqref{it:Omega0}, $\Omega_0^\la-\Omega_0$ is nonzero for all nonzero $\la\in Q_+$.
In $U'(\Fh)$, put
\begin{equation}\label{eq:psi-def}
	\psi_\la = \frac{1}{\Omega_0^\la-\Omega_0},\qquad \la\in Q_{\succ 0}.
\end{equation}
By $\psi_\la^\mu$ we will mean $(\psi_\la)^\mu=\frac{1}{\Omega_0^{\la+\mu}-\Omega_0^\mu}$, the $\mu$-shift of the rational function $\psi_\la\in U'(\Fh)$, with respect to the automorphism $(\cdot)^\mu$ defined in \eqref{eq:shift-def}.
For $\ell\in\mathbb{Z}_{>0}$ and $(\be_1,\ldots,\be_\ell)\in (Q_{\succ 0})^\ell$, we define $E(\be_1,\be_2,\ldots,\be_\ell)\in U'(\Fh)$ by
\begin{equation}\label{eq:E-coefficients}
E(\be_1,\be_2,\ldots,\be_\ell) = \psi_{\be_1+\be_2+\cdots+\be_\ell}\psi^{\be_1}_{\be_2+\cdots+\be_\ell}\cdots\psi^{\be_1+\be_2+\cdots+\be_{\ell-1}}_{\be_\ell}.
\end{equation}
In the following result, we only need $E(\be_1,\ldots,\be_\ell)$ defined for $\be_i\in\De_+$, but to prove the relative ABRR recursion in Section~\ref{sec:the_relative_abrr_recursion} it will be necessary to define \eqref{eq:E-coefficients} for all $\be_i\in Q_{\succ 0}$.

\begin{Lemma}[Resolution of the Identity]\label{lem:E}
For $\la\in Q_{\succ 0}$, let
\begin{equation}\label{eq:E-formula1}
E[\la]=\sum_{\ell=1}^\infty\; \sum_{\substack{(\be_1,\be_2,\ldots,\be_\ell)\in(\De_+)^\ell\\ \be_1+\be_2+\cdots+\be_\ell=\la}} 
E(\be_1,\be_2,\ldots,\be_\ell)
\sum_{i_j\in I_{\be_j}}
f_{\be_1}^{i_1}f_{\be_2}^{i_2}\cdots f_{\be_\ell}^{i_\ell}
e_{\be_\ell}^{i_\ell}\cdots e_{\be_2}^{i_2} e_{\be_1}^{i_1}.
\end{equation}
Then 
\begin{equation}\label{eq:E-property}
	E[\la]\Big|_{M_{-\la}} = \op{Id}_{M_{-\la}}.
\end{equation}
\end{Lemma}

\begin{Remark}
In the case when the root spaces of $\Fg$ are one-dimensional, the sets $I_\be$ are singletons; the summation $\sum_{i_j\in I_{\be_j}}$ and all superscripts $i_j$ can be removed.
\end{Remark}

\begin{proof}
Fix $\la\in Q_{\succ 0}$.
The action of the quadratic Casimir $\Omega$ (see \eqref{eq:casimir}) on a weight vector $v\in M_{-\la}$ is calculated as follows. Write $v=\hat v\boldsymbol{1}$ where $\boldsymbol{1}=1+I$ is the generator of $M$,
and $\hat v\in U'(\Fh)U(\Fg_-)_{-\la}$. Then, since $\Omega$ commutes with the action of $U(\Fg)$ on $M$ and all $e_\be^i$ annihilate $\boldsymbol{1}$, we have
\begin{equation}
	\Omega v = \Omega \hat v\boldsymbol{1} = \hat v \Omega \boldsymbol{1} = \hat v \Omega_0 \boldsymbol{1} \overset{\eqref{eq:shift}}{=} \Omega_0^\la \hat v\boldsymbol{1} = \Omega_0^\la v.
\end{equation}
Subtracting $\Omega_0 v$ from both sides, we can write this as
\begin{equation}
	\sum_{\be\in\De_+}\sum_{i\in I_{\be}}f_\be^i e_\be^i\Big|_{M_{-\la}} = \big(\Omega_0^\la-\Omega_0\big)\op{Id}_{M_{-\la}}.
\end{equation}
Multiplying both sides by $\psi_\la$, defined in \eqref{eq:psi-def}, we find that
\begin{equation}\label{eq:resolution-pf1}
	\sum_{\be_1\in\De_+} \psi_\la \sum_{i_1\in I_{\be_1}}f_{\be_1}^{i_1} e_{\be_1}^{i_1} \Big|_{M_{-\la}} = \op{Id}_{M_{-\la}}.
\end{equation}
Since $e_{\be_1}^{i_1}M_{-\la}\subseteq M_{-\la+\be_1}$, in nonzero terms we have $-\la+\be_1\preceq 0$. Without losing equality, we may thus add the condition $\be_1\preceq \la$ in the sum over $\be_1$.
We then separate the term in which $\be_1=\la$ and the terms in which $\be_1\prec \la$. This gives
\begin{equation}
\op{Id}_{M_{-\la}} = 
\sum_{\substack{\be_1\in\De_+\\ \be_1=\la}} \psi_\la \sum_{i_1\in I_{\be_1}}f_{\be_1}^{i_1} e_{\be_1}^{i_1} \Big|_{M_{-\la}} 
+\sum_{\substack{\be_1\in\De_+\\ \be_1\prec \la}} \psi_\la \sum_{i_1\in I_{\be_1}}f_{\be_1}^{i_1} e_{\be_1}^{i_1} \Big|_{M_{-\la}}.
\end{equation}
In the second sum, since $e_{\be_1}^{i_1}M_{-\la}\subset M_{-(\la-\be_1)}$ and
$\la-\be_1\neq 0$, we insert $\op{Id}_{M_{-(\la-\be_1)}}$ between $f_{\be_1}^{i_1}$ and $e_{\be_1}^{i_1}$ and can use \eqref{eq:resolution-pf1} again (with $\la$ replaced by $\la-\be_1$):
\begin{align}
\op{Id}_{M_{-\la}} &=
\sum_{\substack{\be_1\in\De_+\\ \dot\be_1=\la}}
\psi_{\be_1} \sum_{i_1\in I_{\be_1}}f_{\be_1}^{i_1}e_{\be_1}^{i_1}\Big|_{M_{-\la}} +
\sum_{\substack{(\be_1,\be_2)\in(\De_+)^2\\ \be_1\prec\la}} \psi_\la \sum_{\substack{i_1\in I_{\be_1}\\ i_2\in I_{\be_2}}}
 f_{\be_1}^{i_1}\big(\psi_{\la-\dot\be_1} f_{\be_2}^{i_2}e_{\be_2}^{i_2}\big) e_{\be_1}^{i_1} \Big|_{M_{-\la}} \nonumber\\ &\overset{\eqref{eq:shift}}{=}
\sum_{\substack{\be_1\in\De_+\\ \be_1=\la}} \psi_{\be_1}\sum_{i_1\in I_{\be_1}}f_{\be_1}^{i_1}e_{\be_1}^{i_1}\Big|_{M_{-\la}} +
\sum_{\substack{(\be_1,\be_2)\in(\De_+)^2\\ \be_1\prec\la}}
\psi_\la \psi_{\la-\be_1}^{\be_1} \sum_{i_j\in I_{\be_j}} f_{\be_1}^{i_1} f_{\be_2}^{i_2}e_{\be_2}^{i_2}e_{\be_1}^{i_1} \Big|_{M_{-\la}}. \nonumber 
\end{align}
Let us repeat the process one more time for clarity. In the sum over $(\be_1,\be_2)$, we keep only those terms in which $\be_1+\be_2\preceq \la$, as in the other terms (if any) we have $e_{\be_2}^{i_2}e_{\be_1}^{i_1}M_{-\la}=0$ for reasons of support. Then we separate out the terms in which $\be_1+\be_2=\la$. In the remaining terms, $\be_1+\be_2\prec \la$ and we can insert $\op{Id}_{M_{-(\la-\be_1-\be_2)}}$ in the middle, and use \eqref{eq:resolution-pf1},\eqref{eq:shift} again, giving
\begin{align}
\op{Id}_{M_{-\la}} &=
\sum_{\substack{\be_1\in\De_+\\ \dot\be_1=\la}} \psi_{\be_1} \sum_{i_1\in I_{\be_1}}f_{\be_1}^{i_1}e_{\be_1}^{i_1}\Big|_{M_{-\la}}+
\sum_{\substack{(\be_1,\be_2)\in(\De_+)^2\\ \be_1+\be_2=\la}}
\psi_\la \psi_{\la-\be_1}^{\be_1} \sum_{i_j\in I_{\be_j}} f_{\be_1}^{i_1} f_{\be_2}^{i_2}e_{\be_2}^{i_2}e_{\be_1}^{i_1} \Big|_{M_{-\la}} \\ 
&\quad+\sum_{\substack{(\be_1,\be_2,\be_3)\in(\De_+)^3\\ \be_1+\be_2\prec\la}}
\psi_\la \psi_{\la-\be_1}^{\be_1} \psi_{\la-\be_1-\be_2}^{\be_1+\be_2} \sum_{i_j\in I_{\be_j}}f_{\be_1}^{i_1} f_{\be_2}^{i_2}f_{\be_3}^{i_3}e_{\be_3}^{i_3}e_{\be_2}^{i_2}e_{\be_1}^{i_1} \Big|_{M_{-\la}}.
\end{align}
Repeating this, we obtain
\begin{equation}
\op{Id}_{M_{-\la}} =\sum_{\ell=1}^\infty \; \sum_{\substack{(\be_1,\ldots,\be_\ell)\in(\De_+)^\ell\\ \be_1+\cdots+\be_\ell=\la}}
\psi_\la \psi_{\la-\be_1}^{\be_1} \cdots \psi_{\la-\be_1-\cdots-\be_{\ell-1}}^{\be_1+\cdots+\be_{\ell-1}}
\sum_{i_j\in I_{\be_j}}
f_{\be_1}^{i_1}\cdots f_{\be_\ell}^{i_\ell}e_{\be_\ell}^{i_\ell}\cdots e_{\be_1}^{i_1} \Big|_{M_{-\la}} 
\end{equation}
In each term we may replace $\la$ by $\be_1+\cdots+\be_\ell$. This proves \eqref{eq:E-property}. 
\end{proof}

\begin{Remark}
The assumption that the action of $\Omega$ on $U(\Fg)/U(\Fg)\Fg_+$ is well-defined, implies that $\sum_{\be\in\De_+, i\in I_\be} f_\be^i e_\be^i$ is well-defined on $M$. That is, on each vector in $M$ the result is a finite sum. Therefore $E[\la]$ is also finite on each element of $M$.
\end{Remark}

\begin{Remark}
It is interesting to note that the ``atom'' of the construction, namely
\begin{equation}
	\frac{\Omega-\Omega_0}{\Omega_0^\la-\Omega_0}=\sum_{\be\in\De_+} \psi_\la \sum_{i\in I_\be} f^i_\be e^i_\be
\end{equation}
is invariant under all affine transformations $\Omega \mapsto a\Omega+b$. The divided difference in the LHS is
similar to what appears in \cite[Eq. (26)]{Zh1993} and \cite[Thm.~7]{CS03},\cite[Thm.~3.4]{CS05} in their infinite factorizations of extremal projectors and relative extremal projectors, respectively. 
\end{Remark}

\section{The Resolved Recursion}\label{sec:the_resolved_recursion}
In this section we prove that the relative extremal projector $P_{\Fg,\Fl}$ can be expressed as an infinite sum of symmetric weight zero elements with coefficients from the fraction field $U'(\Fh)$. The coefficients are defined recursively, using the coefficients $E(\be_1,\ldots,\be_\ell)$ in the resolution of the identity $E[\mu]$. For this reason, we call it the \emph{resolved recursion}.
As in the previous section, in Theorem \ref{thm:relative-extremal-projector-sum} below we only need $P_{\Fg,\Fl}(\be_1,\ldots,\be_\ell)$ defined for $\be_i\in\De_+$, but for the statement of the relative ABRR recursion in Section~\ref{sec:the_relative_abrr_recursion} we need to allow $\be_i\in Q_{\succ 0}$.

For $\ell\in\mathbb{Z}_{\ge 0}$ and $(\be_1,\ldots,\be_\ell)\in(Q_{\succ 0})^\ell$, define $P_{\Fg,\Fl}(\be_1,\be_2,\ldots,\be_\ell)\in U'(\Fh)$ recursively as follows:
\begin{subequations}\label{eq:P-coefficients-def}
\begin{enumerate}
	\item When $\ell=0$: 
\begin{equation} \label{eq:resolved-recursion-0}
	P_{\Fg,\Fl}(\emptyset)=1, 
\end{equation}
	where $\emptyset$ denotes the empty sequence.
	\item When $\ell>0$ and $\be_i\in Q_+(\Fl)$ for all $i=1,\ldots,\ell$: 
\begin{equation}\label{eq:resolved-recursion-levi}
	P_{\Fg,\Fl}(\be_1,\ldots,\be_\ell)=0.
\end{equation}
	\item When $\ell>0$ and $(\be_1,\ldots,\be_\ell)\in (Q_+)^\ell$ and at least one $\be_i$ belongs to $Q_+\setminus Q_+(\Fl)$:
\begin{equation}\label{eq:resolved-recursion}
	P_{\Fg,\Fl}(\be_1,\ldots,\be_\ell)
	=-\sum_{k=0}^{\ell-1} P_{\Fg,\Fl}(\be_1,\ldots,\be_k)E(\be_{k+1},\ldots,\be_\ell)^{\be_1+\cdots+\be_k},
\end{equation}
where $E(\be_{k+1},\ldots,\be_\ell)\in U'(\Fh)$ was defined in \eqref{eq:E-coefficients}.
\end{enumerate}
\end{subequations}

\begin{Theorem}\label{thm:relative-extremal-projector-sum}
The relative extremal projector $P_{\Fg,\Fl}$ can be expressed as
\begin{equation}\label{eq:Prel-expansion1}
P_{\Fg,\Fl} = 
1+
\sum_{\ell=1}^\infty\;
\sum_{(\be_1,\ldots,\be_\ell)\in(\De_+)^\ell}
P_{\Fg,\Fl}(\be_1,\ldots,\be_\ell) \sum_{i_j\in I_{\be_j}} f_{\be_1}^{i_1}\cdots f_{\be_\ell}^{i_\ell} e_{\be_\ell}^{i_\ell}\cdots e_{\be_1}^{i_1}\Big|_M.
\end{equation}
\end{Theorem}

\begin{Remark}
In the case when the root spaces of $\Fg$ are one-dimensional, the sets $I_\be$ are singletons; the summation $\sum_{i_j\in I_{\be_j}}$ and all superscripts $i_j$ can be removed.
\end{Remark}

\begin{proof}
Let $\tilde P_{\Fg,\Fl}$ denote the operator in the RHS of \eqref{eq:Prel-expansion1}.
We will show that $\tilde P_{\Fg,\Fl}$ agrees with $P_{\Fg,\Fl}$ on each weight space $M_{-\la}$, $\la\in Q_+$. We write $\tilde P_{\Fg,\Fl}$ as the sum
\begin{equation}
	\tilde P_{\Fg,\Fl} = \sum_{\la\in Q_+} \tilde P_{\Fg,\Fl}[\la]\Big|_M
\end{equation}
where $\tilde P_{\Fg,\Fl}[0]=1$ and for $\la\succ 0$,
\begin{equation}\label{eq:resolved-proof-tilde-P-sum}
	\tilde P_{\Fg,\Fl}[\la] = \sum_{\ell=1}^\infty \;
	\sum_{\substack{(\be_1,\ldots,\be_\ell)\in (\De_+)^\ell \\ \be_1+\cdots+\be_\ell=\la}} P_{\Fg,\Fl}(\be_1,\ldots,\be_\ell) \sum_{i_j\in I_{\be_j}}
	f_{\be_1}^{i_1}\cdots f_{\be_\ell}^{i_\ell} e_{\be_\ell}^{i_\ell}\cdots e_{\be_1}^{i_1}.
\end{equation}
Observe that, by \eqref{eq:support-property}, for all $\la\succeq 0$ we have
\begin{equation}\label{eq:resolved-proof-truncation}
	\tilde P_{\Fg,\Fl}\Big|_{M_{-\la}} = \sum_{0\preceq\mu\preceq\la} \tilde P_{\Fg,\Fl}[\mu]\Big|_{M_{-\la}}.
\end{equation}
First, suppose that $\la\in Q_+(\Fl)$. 
By \eqref{eq:levi-cone-property}, $\be_1+\cdots+\be_\ell\preceq\la$ implies that
$\be_i\in Q_+(\Fl)$ for all $i$. By \eqref{eq:resolved-recursion-levi}, $P_{\Fg,\Fl}(\be_1,\ldots,\be_\ell)=0$ if $\ell>0$. Therefore,
\begin{equation}
	\tilde P_{\Fg,\Fl}\big|_{M_{-\la}} = \tilde P_{\Fg,\Fl}[0]\big|_{M_{-\la}}=\op{Id}_{M_{-\la}}
	\overset{\eqref{eq:relative-extremal-projector-operator-def}}{=} P_{\Fg,\Fl}\big|_{M_{-\la}}.
\end{equation}
It remains to consider the case $\la\in Q_+\setminus Q_+(\Fl)$. By \eqref{eq:relative-extremal-projector-operator-def} and \eqref{eq:resolved-proof-truncation} it suffices to show that 
$\sum_{0\preceq\mu\preceq\la} \tilde P_{\Fg,\Fl}[\mu]\Big|_{M_{-\la}}=0$,
or equivalently, that
\begin{equation}
	\tilde P_{\Fg,\Fl}[\la]\Big|_{M_{-\la}}=-\sum_{0\preceq\mu\prec\la} \tilde P_{\Fg,\Fl}[\mu]\Big|_{M_{-\la}}.
\end{equation}
We insert $\op{Id}_{M_{-(\la-\mu)}}$ between the $f$'s and the $e$'s in the definition \eqref{eq:resolved-proof-tilde-P-sum} of $\tilde P_{\Fg,\Fl}[\mu]$, and use the resolution of the identity, Lemma~\ref{lem:E}. This gives that, as operators on $M_{-\la}$, and suppressing the sums $\sum_{i_j\in I_{\be_j}}$ for brevity:
\begin{align*}
&-\sum_{0\preceq\mu\prec\la} \tilde P_{\Fg,\Fl}[\mu]\\ 
&=-\sum_{0\preceq\mu\prec\la} \sum_{\substack{(\be_1,\ldots,\be_k)\in(\De_+)^k\\ \be_1+\cdots+\be_k=\mu}}
P_{\Fg,\Fl}(\be_1,\ldots,\be_k)f_{\be_1}^{i_1}\cdots f_{\be_k}^{i_k} E[\la-\mu]e_{\be_k}^{i_k}\cdots e_{\be_1}^{i_1} \\ 
&=-\sum_{0\preceq\mu\prec\la} \sum_{\substack{(\be_1,\ldots,\be_\ell)\in(\De_+)^\ell\\ \be_1+\cdots+\be_k=\mu\\\be_{k+1}+\cdots+\be_\ell=\la-\mu}}
P_{\Fg,\Fl}(\be_1,\ldots,\be_k)f_{\be_1}^{i_1}\cdots f_{\be_k}^{i_k}
E(\be_{k+1},\ldots,\be_\ell)f_{\be_{k+1}}^{i_{k+1}}\cdots f_{\be_\ell}^{i_\ell}e_{\be_\ell}^{i_\ell}\cdots e_{\be_1}^{i_1} \\ 
&=\sum_{\ell=1}^\infty\sum_{\substack{(\be_1,\ldots,\be_\ell)\in(\De_+)^\ell\\ \be_1+\cdots+\be_\ell=\la}} \Big(-\sum_{k=0}^{\ell-1} P_{\Fg,\Fl}(\be_1,\ldots,\be_k)E(\be_{k+1},\ldots,\be_\ell)^{\be_1+\cdots+\be_k}\Big)f_{\be_1}^{i_1}\cdots f_{\be_\ell}^{i_\ell}e_{\be_\ell}^{i_\ell}\cdots e_{\be_1}^{i_1}.
\end{align*}
By \eqref{eq:resolved-proof-tilde-P-sum} and the recursion \eqref{eq:resolved-recursion}, this equals $P[\mu]\Big|_{M_{-\la}}$.
\end{proof}

\section{The Relative ABRR Recursion}\label{sec:the_relative_abrr_recursion}

Relative extremal projectors were introduced in \cite{CS03}.
In this section we provide a generalization of the ABRR recursion \cite{ABRR,Khoroshkin2004} to the case of relative extremal projectors.

The rational functions $\psi_\la$ defined in \eqref{eq:psi-def} satisfy the following key identity:
\begin{Lemma}[Cocycle Identity for $\psi_\la$]
\begin{equation}\label{eq:cocycle}
	(\psi_\la-\psi_{\la+\mu})\psi_\mu^\la = \psi_\la\psi_{\la+\mu},\qquad\forall \la,\mu\in Q_{\succ 0}.
\end{equation}
\end{Lemma}
\begin{proof}
Straightforward:
\[
\psi_\la-\psi_{\la+\mu} 
= \frac{1}{\Omega_0^\la-\Omega_0}-\frac{1}{\Omega_0^{\la+\mu}-\Omega_0} 
= \frac{\Omega_0^{\la+\mu}-\Omega_0^\la}{(\Omega_0^\la-\Omega_0)(\Omega_0^{\la+\mu}-\Omega_0)} = \frac{\psi_\la\psi_{\la+\mu}}{\psi^\la_\mu}.
\]
\end{proof}

For $\be\in Q_+$, we use the following generalized Kronecker delta notation:
\begin{equation}
\delta_{\be\in Q_+(\Fl)}=
\begin{cases}
1,&\be\in Q_+(\Fl), \\ 
0,&\text{otherwise}.
\end{cases}
\end{equation}

The following is the second main theorem of the paper:

\begin{Theorem}\label{thm:relative-ABRR}
The rational functions $P_{\Fg,\Fl}(\be_1,\be_2,\ldots,\be_\ell)$ defined in 
\eqref{eq:P-coefficients-def}, satisfy
\begin{align}\label{eq:relative-ABRR-initial0}
P_{\Fg,\Fl}(\emptyset)&=1,\\
\label{eq:relative-ABRR-initial1}
P_{\Fg,\Fl}(\beta)&=(-1+\delta_{\be\in Q_+(\Fl)})\psi_\be,\qquad \forall\be\in Q_{\succ 0}, \\ 
\intertext{and for all integers $\ell\ge 2$ and all $\be_1,\ldots,\be_\ell\in Q_{\succ 0}$:}
\label{eq:relative-ABRR-recursion}
	P_{\Fg,\Fl}(\be_1,\ldots,\be_\ell)&=
	-\psi_{\be_1}P_{\Fg,\Fl}(\be_1+\be_2,\be_3,\ldots,\be_\ell)
	+\delta_{\be_1\in Q_+(\Fl)}\psi_{\be_1}P_{\Fg,\Fl}(\be_2,\ldots,\be_\ell)^{\be_1}.
\end{align}
\end{Theorem}

\begin{proof}
By \eqref{eq:P-coefficients-def}, the case $\ell=0$ is trivial.
We will prove \eqref{eq:relative-ABRR-recursion} by induction on $\ell\ge 1$, regarding \eqref{eq:relative-ABRR-initial1} as the correct interpretation when $\ell=1$.
For $\ell=1$ and $\be_1=\be$ we have
\begin{equation}
	P_{\Fg,\Fl}(\be) =
	\begin{cases}
	0,& \be\in Q_+(\Fl), \\ 
	-E(\be),& \be\notin Q_+(\Fl).
	\end{cases}
\end{equation}
By \eqref{eq:E-coefficients}, $E(\be)=\psi_\be$. Therefore,
\begin{equation}
	P_{\Fg,\Fl}(\be)=(-1+\delta_{\be\in Q_+(\Fl)})\psi_\be.
\end{equation}
Now suppose $\ell>1$.
Before proceeding, note that if $\be_i$ belong to the levi cone $Q_+(\Fl)$ for all $i=1,2,\ldots,\ell$, then both sides of \eqref{eq:relative-ABRR-recursion} are zero, by \eqref{eq:resolved-recursion-levi}. In the remainder of the proof, we may therefore assume that at least one $\be_i$ is not in $Q_+(\Fl)$. By the resolved recursion \eqref{eq:resolved-recursion}, we then have
\begin{equation}\label{eq:rel_rec2_proof-1}
P_{\Fg,\Fl}(\be_1,\ldots,\be_\ell)
=-\sum_{k=0}^{\ell-1}P_{\Fg,\Fl}(\be_1,\ldots,\be_k)E(\be_{k+1},\ldots,\be_\ell)^{\be_1+\cdots+\be_k}.
\end{equation}
Separating out the $k=0$ term, and using the induction hypothesis in the remaining terms, the right hand side of \eqref{eq:rel_rec2_proof-1} equals
\begin{align*}
&-E(\be_1,\ldots,\be_\ell) \\ 
&-\sum_{k=1}^{\ell-1}\Big(-\psi_{\be_1}P_{\Fg,\Fl}(\be_1+\be_2,\be_3,\ldots,\be_k)+\delta_{\be_1\in Q_+(\Fl)}\psi_{\be_1}P_{\Fg,\Fl}(\be_2,\ldots,\be_k)^{\be_1}\Big)\cdot \\
&\qquad\qquad \cdot E(\be_{k+1},\ldots,\be_\ell)^{\be_1+\cdots+\be_k},
\end{align*}
where $P_{\Fg,\Fl}(\cdots)$ should both be interpreted as $1$ when $k=1$.
Rearranging terms, we write this as
\begin{align} 
&-E(\be_1,\ldots,\be_\ell) + \psi_{\be_1}E(\be_2,\ldots,\be_\ell)^{\be_1}\label{term1}\\ 
&+\psi_{\be_1}\sum_{k=2}^{\ell-1} P_{\Fg,\Fl}(\be_1+\be_2,\be_3,\ldots,\be_k) E(\be_{k+1},\ldots,\be_\ell)^{\be_1+\cdots+\be_k} \label{term2} \\ 
&+\delta_{\be_1\in Q_+(\Fl)}\psi_{\be_1}\Big[-\sum_{k=1}^{\ell-1}
P_{\Fg,\Fl}(\be_2,\ldots,\be_k)
\cdot E(\be_{k+1},\ldots,\be_\ell)^{\be_2+\cdots+\be_k}\Big]^{\be_1}.
\label{term3}
\end{align}
The two terms in \eqref{term1} can combined, using the cocycle identity \eqref{eq:cocycle} and \eqref{eq:E-coefficients}:
\begin{equation}\label{eq:cocycle-E}
	-E(\be_1,\be_2,\ldots,\be_\ell)+\psi_{\be_1}E(\be_2,\ldots,\be_\ell)^{\be_1}=\psi_{\be_1}E(\be_1+\be_2,\be_3,\ldots,\be_\ell).
\end{equation}
Adding this to \eqref{term2}, and using that at least one of $(\be_1+\be_2), \be_3, \ldots,\be_\ell$ is not in $Q_+(\Fl)$ (if $\be_1+\be_2$ is in $Q_+(\Fl)$ then both $\be_1$ and $\be_2$ is in $Q_+(\Fl)$), we can use the resolved recursion \eqref{eq:resolved-recursion} and obtain
\begin{equation}\label{eq:term1plusterm2}
\eqref{term1}+\eqref{term2}=-\psi_{\be_1}P_{\Fg,\Fl}(\be_1+\be_2,\be_3,\ldots,\be_\ell).
\end{equation}
Lastly, in \eqref{term3}, if there is an $i\in\{2,3,\ldots,\ell\}$ such that $\be_i\notin Q_+(\Fl)$, we may use the resolved recursion \eqref{eq:resolved-recursion} to conclude
\begin{equation}\label{eq:term3-equality}
	\eqref{term3}=\delta_{\be_1\in Q_+(\Fl)}\psi_{\be_1}P_{\Fg,\Fl}(\be_2,\be_3,\ldots,\be_\ell)^{\be_1}.
\end{equation}
If on the other hand $\be_i\in Q_+(\Fl)$ for $i=2,3,\ldots,\ell$, then $\be_1$ must be the one that is not in $Q_+(\Fl)$. But then both sides of \eqref{eq:term3-equality} are zero, so equality holds in this case too.

Adding \eqref{eq:term1plusterm2} to \eqref{eq:term3-equality}, we obtain the right hand side of \eqref{eq:relative-ABRR-recursion}.
\end{proof}

\section{An Explicit Formula for the Coefficients}\label{sec:an_explicit_formula_for_the_coefficients}
\begin{Theorem}\label{thm:explicit}
Let $\ell$ be a positive integer, and let $\be_1,\be_2,\ldots,\be_\ell\in Q_{\succ 0}$.
If all $\be_i\in Q_+(\Fl)$ then $P_{\Fg,\Fl}(\be_1,\ldots,\be_\ell)=0$.
Suppose that $\be_1,\be_2,\ldots,\be_{m-1}\in Q_+(\Fl)$, while $\be_m\notin Q_+(\Fl)$, where $m\in\{1,2,\ldots,\ell\}$. Then
\begin{align}\label{eq:explicit-solution}
P_{\Fg,\Fl}(\be_1,\ldots,\be_\ell) &=\sum_{k=0}^{m-1} E(\be_1,\ldots,\be_k)P_\Fg(\be_{k+1},\ldots,\be_\ell)^{\be_1+\cdots+\be_k} \\ 
&=\sum_{k=0}^{m-1}(-1)^{\ell-k} \psi_{\be_1+\cdots+\be_k}\psi_{\be_2+\cdots+\be_k}^{\be_1}\cdots\psi_{\be_k}^{\be_1+\cdots+\be_{k-1}} \nonumber \\
& \qquad\cdot \psi_{\be_{k+1}}^{\be_1+\cdots+\be_k}\psi_{\be_{k+1}+\be_{k+2}}^{\be_1+\cdots+\be_k}\cdots \psi_{\be_{k+1}+\cdots+\be_\ell}^{\be_1+\cdots+\be_k}.
\label{eq:explicit-solution2}
\end{align}
\end{Theorem}

\begin{proof}
The first part is just a restatement of \eqref{eq:resolved-recursion-levi}. Therefore we will assume that at least one $\be_i$ is not in $Q_+(\Fl)$.
First, we consider the case when $m=1$, meaning $\be_1\notin Q_+(\Fl)$. Then the right hand side just equals $P_\Fg(\be_1,\ldots,\be_\ell)$ which coincides with $P_{\Fg,\Fl}(\be_1,\ldots,\be_\ell)$ when $\be\notin Q_+(\Fl)$, by repeated use of the ABRR recursion \eqref{eq:relative-ABRR-recursion}.
In the rest of the proof we may therefore assume that $\be_1\in Q_+(\Fl)$, and $\ell \ge 2$.

We proceed by induction on $\ell$. By the relative ABRR recursion \eqref{eq:relative-ABRR-recursion}, we have
\begin{equation}\label{eq:solution_pf1}
	P_{\Fg,\Fl}(\be_1,\ldots,\be_\ell) = -\psi_{\be_1} P_{\Fg,\Fl}(\be_1+\be_2,\ldots,\be_\ell)+\psi_{\be_1} P_{\Fg,\Fl}(\be_2,\ldots,\be_\ell)^{\be_1}.
\end{equation}
Since $\be_1,\be_2,\ldots,\be_{m-1}\in Q_+(\Fl)$, the same is true for $\be_1+\be_2$.
By the induction hypothesis, the RHS of \eqref{eq:solution_pf1} thus equals
\begin{multline}
-\psi_{\be_1}\sum_{k=1}^{m-1} E(\be_1+\be_2,\ldots,\be_k)P(\be_{k+1},\ldots,\be_\ell)^{\be_1+\cdots+\be_k} \\ 
+\psi_{\be-1}\sum_{k=1}^{m-1} E(\be_2,\ldots,\be_k)^{\be_1}P(\be_{k+1},\ldots,\be_\ell)^{\be_1+\be_2+\cdots+\be_k}
\end{multline}
By the cocycle identity \eqref{eq:cocycle-E}, this equals
\begin{equation}
	\sum_{k=0}^{m-1} E(\be_1,\ldots,\be_k)P(\be_{k+1},\ldots,\be_\ell)^{\be_1+\cdots+\be_k}.
\end{equation}
This proves \eqref{eq:explicit-solution}. The second equality \eqref{eq:explicit-solution2} follows from \eqref{eq:E-coefficients} and the formula for $P_\Fg(\be_1,\ldots,\be_\ell)$ obtained by repeated use of the relative ABRR recursion \eqref{eq:relative-ABRR-recursion} (the delta term is always zero for $\Fl=\Fh$). (See also Section \ref{sec:non-relative-case}.)
\end{proof}

\begin{Remark}\label{rem:alternative-route-to-solution}
The way we discovered the solution \eqref{eq:explicit-solution} was through a different route, inspired by \cite[Thm.~8]{CS03}. It relies on already knowing that the usual extremal projector $P_\Fg$ can be given in the form \eqref{eq:Prel-expansion1}. Although this was known \cite{Khoroshkin2004} (at least in the semisimple case), we chose to follow a self-contained logic in our main arguments above. Nevertheless, the alternative path is also interesting and therefore we present it here:

We go back to the ``resolution of the identity on $M_{-\la}$'' (let us for the sake of readability omit the $\sum_{i_j\in I_\be}$ and superscripts $i_j$):
\begin{equation}
	E[\la]\overset{\eqref{eq:E-formula1}}{=}\sum_{\ell=1}^\infty \sum_{\substack{(\be_1,\ldots,\be_\ell)\in(\De_+)^\ell\\\be_1+\cdots+\be_\ell=\la}}E(\be_1,\ldots,\be_\ell) f_{\be_1}\cdots f_{\be_\ell}e_{\be_\ell}\cdots e_{\be_1}.
\end{equation}
The property we proved was that it is the identity on $M_{-\la}$. Following the idea in \cite[Lem.~3]{CS03},
and also \cite[Eq. (24)]{ForQuiSki2025} (where the extremal projector is denoted denoted $|h\rangle \langle h|$),
we may insert the (non-relative) extremal projector $P_\Fg$ in the middle, and then sum over all $\la$, and include the constant term $1$. The result is now the identity on all of $M$. So we obtain the \emph{resolution of the identity on $M$}:
\begin{equation}
	1 = \sum E(\be_1,\ldots,\be_\ell) f_{\be_1}\cdots f_{\be_\ell} P_\Fg e_{\be_\ell}\cdots e_{\be_1}.
\end{equation}
where we sum over all finite sequences of positive roots $\be_i\in \De_+$, including the empty sequence for which the term is just $P_\Fg$. Now, replace $\Fg$ by $\Fl$ in this formula. It is then a sum over sequences of $\be_i\in \De_+(\Fl)$ and has a $P_\Fl$ in the middle. Then, as in \cite[Thm.~8]{CS03}, multiply both sides by $P_{\Fg,\Fl}$, and use that it commutes with every operator from $U(\Fl)$, and that $P_{\Fg,\Fl}P_{\Fl}=P_\Fg$, and we obtain:
\begin{equation}\label{eq:another-formula-for-Pgl}
	P_{\Fg,\Fl} = \sum E(\be_1,\ldots,\be_k) f_{\be_1}\cdots f_{\be_k} P_\Fg e_{\be_k}\cdots e_{\be_1}
\end{equation}
where the sum is over all finite sequences of positive roots for $\Fl$.
We remark that \eqref{eq:another-formula-for-Pgl} is a direct generalization of \cite[Cor.~1]{CS03}.
Now use the explicit ABRR-type form of the (non-relative) extremal projector $P_\Fg$ first obtained in \cite{Khoroshkin2004}:
\begin{equation}
	P_\Fg = \sum P_\Fg(\be_1,\ldots,\be_\ell) f_{\be_1}\cdots f_{\be_\ell}e_{\be_\ell}\cdots e_{\be_1}
\end{equation}
summing over all sequences of positive roots. Then we get
\begin{equation}
	P_{\Fg,\Fl}=\sum E(\be_1,\ldots,\be_k)f_{\be_1}\ldots f_{\be_k} P(\be_{k+1},\ldots,\be_{\ell})f_{\be_{k+1}}\cdots f_{\be_\ell}e_{\be_\ell}\cdots e_{\be_1}.
\end{equation}
Here the sum is over all finite sequences $(\be_1,\ldots,\be_k)\in(\De_+(\Fl))^k$ and finite sequences $(\be_{k+1},\ldots,\be_\ell)\in (\De_+)^{\ell-k}$. We can move the $P$ coefficients to the left if we apply a shift. Then we can collect the terms corresponding to a sequence $(\be_1,\ldots,\be_\ell)\in (\De_+)^\ell$. The result is that
\begin{equation}
	P_{\Fg,\Fl}(\be_1,\ldots,\be_\ell)=\sum_{k=0}^{m-1} E(\be_1,\ldots,\be_k)P_\Fg(\be_{k+1},\ldots,\be_\ell)^{\be_1+\cdots+\be_k}
\end{equation}
if $\be_1,\ldots,\be_{m-1}\in\De_+(\Fl)$ and $\be_m\notin\De_+(\ell)$. This agrees with \eqref{eq:explicit-solution}.
\end{Remark}

\section{Examples} \label{sec:examples}
\subsection{Explicit formulas for some sequences}
Let $\ell\in\mathbb{Z}_{\ge 1}$ and $\be_1,\ldots,\be_\ell\in Q_{\succ 0}$.
\begin{enumerate}[{\rm 1.}]
\item If $\be_1\notin Q_+(\Fl)$ then $\be_1+\cdots+\be_k\notin Q_+(\Fl)$ for all $k$, so using \eqref{eq:relative-ABRR-recursion} repeatedly,
\begin{align}
P_{\Fg,\Fl}(\be_1,\ldots,\be_\ell) \label{eq:when-relative-equals-non-relative}
&=(-1)^\ell \psi_{\be_1}\psi_{\be_1+\be_2}\cdots\psi_{\be_1+\be_2+\cdots+\be_\ell}\\
&=P_\Fg(\be_1,\ldots,\be_\ell).
\end{align}
\item If $\be_1\in Q_+(\Fl), \be_2\notin Q_+(\Fl)$, then by \eqref{eq:relative-ABRR-recursion} and the previous case we obtain
\begin{align}
	&P_{\Fg,\Fl}(\be_1,\be_2,\ldots,\be_\ell) =\psi_{\be_1}P_\Fg(\be_2,\ldots,\be_\ell)^{\be_1}
	-\psi_{\be_1}P_\Fg(\be_1+\be_2,\be_3,\ldots,\be_\ell) \\  
	&=(-1)^{\ell-1}\psi_{\be_1}\psi_{\be_2}^{\be_1}\psi_{\be_2+\be_3}^{\be_1}\cdots \psi_{\be_2+\be_3+\cdots+\be_\ell}^{\be_1}
	+(-1)^{\ell}\psi_{\be_1}\psi_{\be_1+\be_2}\cdots\psi_{\be_1+\be_2+\cdots+\be_\ell}.
\end{align}
Alternatively, this is the result of taking $m=2$ in \eqref{eq:explicit-solution2}.
\item If $\be_1,\ldots,\be_{\ell-1}\in Q_+(\Fl)$ and $\be_\ell\notin Q_+(\Fl)$, the resolved recursion \eqref{eq:resolved-recursion}, and the fact that $P_{\Fg,\Fl}(\be_1,\ldots,\be_k)=0$ for all $k=1,2,\ldots,\ell-1$ (by \eqref{eq:resolved-recursion-levi}) gives a single term:
\begin{align}
	P_{\Fg,\Fl}(\be_1,\be_2,\ldots,\be_\ell) &= -E(\be_1,\be_2,\ldots,\be_\ell) \\ 
	&\overset{\eqref{eq:E-coefficients} }{=}-\psi_{\be_1+\be_2+\cdots+\be_\ell}\psi_{\be_2+\cdots+\be_\ell}^{\be_1}\cdots\psi_{\be_\ell}^{\be_1+\be_2+\cdots+\be_{\ell-1}}.
\end{align}
On the other hand, using the relative ABRR recursion \eqref{eq:relative-ABRR-recursion} in this case would lead to a more involved expression. Likewise for \eqref{eq:explicit-solution2} with $m=\ell$. This illustrates that the resolved recursion \eqref{eq:resolved-recursion} and the relative ABRR recursion \eqref{eq:relative-ABRR-recursion} play complementary roles. Both are recursions for the same family of rational functions, $P_{\Fg,\Fl}(\be_1,\ldots,\be_\ell)$, but their immediate application lead to different looking expressions. Those expressions are ultimately equal, due to the cocycle identity \eqref{eq:cocycle} for $\psi_\la$.
\end{enumerate}

\subsection{Finite-dimensional Reductive Lie Algebras}
\label{sec:example-reductive}

Let $\Fg$ be a finite-dimensional complex reductive Lie algebra. For convenience of the reader, we provide details (see e.g. \cite[Eq. (4.1)]{Khoroshkin2004}) for computing the key quantity $\psi_\la$ in this case. Choose
\begin{itemize}
\item a non-degenerate invariant symmetric bilinear form $(\cdot,\cdot)$ on $\Fg$;
\item a Cartan subalgebra $\Fh$ of $\Fg$;
\item a polarization of the root system: $\De=\De_+\sqcup (-\De_+)$, $(\De_++\De_+)\cap\De\subseteq\De_+$;
\item a nonzero root vector $e_\be\in \Fg_\be$ for each $\be\in\De_+$;
\item a basis $\{h_i\}_{i=1}^n$ for $\Fh$.
\end{itemize}
For example, on the semisimple part $[\Fg,\Fg]$ we can choose the Killing form, and on the center of $\Fg$ any non-degenerate symmetric bilinear form will do. Or, if $\Fg$ is a matrix Lie algebra one can usually use a trace form $(x,y)=\op{Tr}(xy)$.
By invariance of the form, for $\al,\be\in\De$ we have $(\Fg_\al,\Fg_\be)\neq 0$ iff $\al+\be=0$. Since the root spaces are one-dimensional, there is for each $\be\in\De_+$ a unique $f_\be\in \Fg_{-\be}$ with 
\begin{equation}
(e_\be,f_\be)=1.	
\end{equation}
Similarly, let $\{h^i\}_{i=1}^n$ denote the basis in $\Fh$ dual to $\{h_i\}_{i=1}^n$ with respect to the form:
\begin{equation}
(h^i,h_j)=\delta^i_j=\begin{cases}1,&i=j,\\0, &i\neq j.\end{cases}	
\end{equation}
Define $\la\mapsto h_\la$ to be the inverse of the map $\Fh\to\Fh^\ast, a\mapsto (a,\cdot)$. Equivalently,
\begin{equation}\label{eq:cartan-dual}
(h_\la,a)=\la(a),\quad\forall \la\in\Fh^\ast, a\in \Fh.	
\end{equation}
By invariance of the form, $(h,[e_\be,f_\be])=([h,e_\be],f_\be)=\be(h)(e_\be,f_\be)=\be(h)$ for all $h\in\Fh$, which means that
\begin{equation}\label{eq:ef-commutator}
	[e_\be,f_\be]=h_\be.
\end{equation}
Note that $(e_\be,f_\be,h_\be)$ need not be an $\Fsl_2$-triple, as we have not normalized the form (and we do not need to). Therefore the $h_\be$ should not be confused with the coroots.
Since $\{x_j\}_{j=1}^{\dim\Fg}:=\{h_i\}_{i=1}^n\cup\{e_\be\}_{\be\in\De_+}\cup\{f_\be\}_{\be\in\De_+}$ and $\{x^j\}_{j=1}^{\dim\Fg}:=\{h^i\}_{i=1}^n\cup\{f_\be\}_{\be\in\De_+}\cup\{e_\be\}_{\be\in\De_+}$ are dual bases for $\Fg$ with respect to the pairing $(\cdot,\cdot)$, one half times the quadratic Casimir associated with $(\cdot,\cdot)$ equals
 \begin{equation}
 	\Omega = \frac{1}{2}\sum_j x^j x_j = \frac{1}{2}\sum_{i=1}^n h^i h_i +\frac{1}{2}\sum_{\be\in\De_+} (f_\be e_\be + e_\be f_\be) = \Omega_0+\sum_{\be\in \De_+} f_\be e_\be.
 \end{equation}
Here we used $e_\be f_\be=f_\be e_\be + h_\be$ by \eqref{eq:ef-commutator} and introduced
\begin{equation}
	\Omega_0 = \frac{1}{2}\sum_{\be\in\De_+} h_\be + \frac{1}{2}\sum_{i=1}^n h^ih_i = h_\rho + \frac{1}{2}\sum_{i=1}^n h^ih_i,
\end{equation}
where $\rho=\frac{1}{2}\sum_{\be\in\De_+}\be$ is the Weyl vector. 
For $\la\in\Fh^\ast$, the $\la$-shift automorphism defined in \eqref{eq:shift-def} applied to $\Omega_0$ equals
\begin{equation}
	\Omega_0^\la = h_\rho + \la(h_\rho) + \frac{1}{2}\sum_i \big(h^i+\la(h^i)\big)\big(h_i+\la(h_i)\big).
\end{equation}
Therefore,
\begin{align}
	\Omega_0^\la-\Omega_0 &= \la(h_\rho)+ \frac{1}{2}\sum_i \Big(\la(h^i)h_i + \la(h_i)h^i + \la(h^i)\la(h_i)\Big) \nonumber \\ 
	&= \la(h_\rho) + \frac{1}{2}(\la,\la) + \sum_i \la(h^i)h_i\nonumber\\
	&=(\la,\rho) + \frac{1}{2}(\la,\la) + h_\la,
\end{align}
using $h=(h^i,h)h_i$ for $h\in\Fh$ and \eqref{eq:cartan-dual}, and we set $(\la,\mu)=(h_\la,h_\mu)$ for $\la,\mu\in\Fh^\ast$.
Thus we have the following explicit form of the quantities $\psi_\la$:
\begin{equation}\label{eq:psi-reductive}
	\psi_\la = \frac{1}{\Omega_0^\la-\Omega_0}=\frac{1}{h_\la + (\la, \rho) + \frac{1}{2}(\la,\la)}.
\end{equation}
We can also observe that since $(h_\la)^\mu=h_\la + \mu(h_\la) = h_\la+(\mu,\la)$ we have
\begin{equation}\label{eq:psi-shift-reductive}
	\psi_\la^\mu = \frac{1}{h_\la + (\mu,\la) + (\rho,\la)+ \frac{1}{2}(\la,\la)}.
\end{equation}
Lastly, for the relative case, to choose a $\De_+(\Fl)$ satisfying \eqref{eq:levi-cone-property} is equivalent to selecting a subset $\Pi(\Fl)$ of the simple roots of $\Fg$ and taking $\De_+(\Fl) = \De_+\cap \mathbb{Z}_{\ge 0}\Pi(\Fl)$.

\subsection{Contragredient Lie Superalgebras}
In the terminology of \cite{Serganova2010} or \cite{Rao2021}, let $\Fg(A)$ be the contragredient Lie superalgebra associated to a symmetrizable, admissible square matrix $A=(a_{ij})_{i,j\in I}$ where $I=I_{\bar 0}\sqcup I_{\bar 1}$ is a finite index set partitioned into even and odd subsets. Then $\Fg(A)$ has a Casimir operator $\Omega$ of the required form \eqref{eq:casimir}. That $A$ is admissible means that the adjoint action on $\Fg(A)$ of the simple root vectors is locally nilpotent. This ensures that the action of $\Omega$ is well-defined on the universal Verma module. The $\Omega_0$ part has the form $\Omega_0=h_\rho+\frac{1}{2}\sum_i h^i h_i$ just like in the reductive case, except that the Weyl vector $\rho\in\Fh^\ast$ is (instead of the half sum of positive roots) chosen so that $2(\rho,\al_i)=(\al_i,\al_i)$ for all simple roots $\al_i$. The assumptions of Section \ref{sec:setup} hold. Consequently, all the results about relative extremal projectors covered in the previous sections go through. See also \cite[Prop.~2.4]{Rao2021} specializing this to the basic classical Lie superalgebras.

\subsection{The Non-Relative Case}\label{sec:non-relative-case}
When we choose $\De_+(\Fl)=\emptyset$, we have $\Fl=\Fh$ and $Q_+(\Fl)=\{0\}$. Thus $\delta_{\be\in Q_+(\Fh)}=0$ for all $\be\in Q_{\succ 0}$. Hence, the recursion \eqref{eq:relative-ABRR-recursion} becomes
\begin{equation}
	P_\Fg(\be_1,\be_2,\ldots,\be_\ell) = -\psi_{\be_1}P_{\Fg}(\be_1+\be_2,\be_3,\ldots,\be_\ell).
\end{equation}
This is the classical ABRR recursion for the extremal projector \cite[Thm.~2 and Cor.~1]{Khoroshkin2004},\cite[Eq.~(5.4)]{KhoOgi2011}, but stated in terms of the coefficients. This recursion can of course be explicitly solved:
\begin{equation}\label{eq:non-relative-projector-coefficients}
	P_\Fg(\be_1,\be_2,\ldots,\be_\ell) = (-1)^\ell \psi_{\be_1}\psi_{\be_1+\be_2}\cdots\psi_{\be_1+\be_2+\cdots+\be_\ell}.
\end{equation}
This gives the following result.
\begin{Corollary} Let $\Fg$ be a Lie superalgebra satisfying the assumptions in Section~\ref{sec:setup}. Then the extremal projector $P_\Fg$ for $\Fg$ may be expanded as
\begin{equation}
	P_\Fg = 1+\sum_{\ell=1}^\infty \sum_{(\be_1,\ldots,\be_\ell)\in(\De_+)^\ell} 
	P_\Fg(\be_1,\ldots,\be_\ell) \sum_{i_j\in I_{\be_j}} f_{\be_1}^{i_1}\cdots f_{\be_\ell}^{i_\ell}e_{\be_\ell}^{i_\ell}\cdots e_{\be_1}^{i_1}
\end{equation}
where $P_\Fg(\be_1,\ldots,\be_\ell)$ are given by \eqref{eq:non-relative-projector-coefficients}.
\end{Corollary}

If we further specialize to the case of a finite-dimensional reductive Lie algebra over $\mathbb{C}$, or a (finite-dimensional) basic classical Lie superalgebra over $\mathbb{C}$, then all root spaces are one-dimensional, and the formula simplifies to
\begin{equation}
	P_\Fg = 1+\sum_{\ell=1}^\infty\; \sum_{(\be_1,\ldots,\be_\ell)\in(\De_+)^\ell} 
	P_\Fg(\be_1,\ldots,\be_\ell) f_{\be_1}\cdots f_{\be_\ell}e_{\be_\ell}\cdots e_{\be_1}.
\end{equation}
This agrees with \cite{Khoroshkin2004} for purely even $\Fg$.


\begin{thebibliography}{MM}
\bibitem{ABRR}
	{\sc D. Arnaudon, E. Buffenoir, E. Ragoucy, Ph. Roche},
	{\em Universal Solutions of Quantum Dynamical Yang–Baxter Equations},
	Letters in Mathematical Physics 44, 201--214 (1998).
	\href{https://doi.org/10.1023/A:1007498022373}{doi:10.1023/A:1007498022373}
\bibitem{AST1971}
	{\sc R. M. Asherova, Yu. F. Smirnov, V. N. Tolstoy},
	{\em Projection Operators for Simple Lie Groups},
	Teoreticheskaya i Matematicheskaya Fizika, Vol. 8, No. 2, 1971, 255--271.
\bibitem{Clebsch1862}
	{\sc A. Clebsch},
	{\em Ueber eine Eigenschaft der Kugelfunctionen},
	Journal für die reine und angewandte Mathematik,
	vol. 1862, no. 60, 1862, pp. 343--350.
	\href{https://doi.org/10.1515/crll.1862.60.343}{doi:10.1515/crll.1862.60.343}
\bibitem{CS03} 
	{\sc C. Conley, M. Sepanski},
	{\em Relative extremal projectors},
	Adv. Math. 174(2) (2003), 155--166.
\bibitem{CS05}
	{\sc C. Conley, M. Sepanski},
	{\em Infinite commutative product formulas for relative extremal projectors},
	Adv. Math. 196(1) (2005), 52--77.
\bibitem{CS15}
	{\sc C. Conley, M. Sepanski},
	{\em Factorizations of relative extremal projectors},
	p-Adic Numbers Ultrametric Anal. Appl. 7 no. 4 (2015), 276--290.
	\href{https://doi.org/10.1134/S2070046615040044}{doi:10.1134/S2070046615040044},
	\href{https://doi.org/10.48550/arXiv.1507.02587}{arXiv:1507.02587}
\bibitem{EV1999}
	{\sc P. Etingof, A. Varchenko},
	{\em Exchange Dynamical Quantum Groups},
	Commun. Math. Phys. 205, 19--52 (1999).
\bibitem{FH2025}
	{\sc D. Fillmore, J. T. Hartwig},
	{\em Subcategories of Module Categories via Restricted Yoneda Embeddings},
	arXiv:2507.12778 [math.RT]. 
	\href{https://doi.org/10.48550/arXiv.2507.12778}{doi:10.48550/arXiv.2507.12778}
\bibitem{ForQuiSki2025}
	{\sc J.-F. Fortin, L. Quintavalle, W. Skiba},
	{\em Virasoro completeness relation and the inverse Shapovalov form},
	Phys. Rev. D 111, 085010, 7 April, 2025.
	\href{https://doi.org/10.1103/PhysRevD.111.085010}{doi:10.1103/PhysRevD.111.085010}
\bibitem{HW2022}
	{\sc J. T. Hartwig, D. A. Williams II},
	{\em Diagonal reduction algebra for $\mathfrak{osp}(1|2)$},
	Theor Math Phys 210, 155--171 (2022).
	\href{https://doi.org/10.1134/S0040577922020015}{doi:10.1134/S0040577922020015}
\bibitem{Khoroshkin2004}
	{\sc S. M. Khoroshkin},
	{\em Extremal Projector and Dynamical Twist}
	Theoretical and Mathematical Physics 139(1): 582--597 (2004).
\bibitem{KhoOgi2008}
	{\sc S. Khoroshkin, O. Ogievetsky},
	{\em Mickelsson algebras and Zhelobenko operators},
	Journal of Algebra 319 (2008) 2113--2165.
\bibitem{KhoOgi2011}
	{\sc S. Khoroshkin, O. Ogievetsky},
	{\em Structure Constants of Diagonal Reduction Algebras of gl Type},
	Symmetry, Integrability and Geometry: Methods and Applications SIGMA 7 (2011), 064, 34 pages.
	\href{https://doi.org/10.3842/SIGMA.2011.064}{doi:10.3842/SIGMA.2011.064}
\bibitem{Low1958}
	{\sc P.-O. Löwdin},
	{\em Angular momentum wave functions constructed by projection operators},
	Technical Note No. 12, 
	Uppsala Univ.(Sweden). Quantum Chemistry Group, 1958.
\bibitem{Low1964}
	{\sc P.-O. Löwdin},
	{\em Angular Momentum Wavefunctions Constructed by Projector Operators},
	Rev. Mod. Phys. 36, 966, 1964.
	\href{https://doi.org/10.1103/RevModPhys.36.966}{doi:10.1103/RevModPhys.36.966}
\bibitem{Rao2021}
	{\sc S. Eswara Rao},
	{\em Generalized Casimir operators for Lie superalgebras},
	J. Math. Phys. 62, 101703 (2021),
	\href{https://doi.org/10.1063/5.0056538}{doi:10.1063/5.0056538}
\bibitem{Serganova2010}
	{\sc V. Serganova},
	{\em Kac-Moody Superalgebras and Integrability},
	in: Neeb, KH., Pianzola, A. (eds) Developments and Trends in Infinite-Dimensional Lie Theory.
	Progress in Mathematics, vol 288. Birkhäuser Boston, 2011.
	\href{https://doi.org/10.1007/978-0-8176-4741-4_6}{doi:10.1007/978-0-8176-4741-4\_6}
\bibitem{Shapiro1965}
	{\sc J. Shapiro},
	{\em Matrix Representation of the Angular Momentum Projection
Operator},
	J. Math. Phys. 6, (1965), 1680--1691.
\bibitem{Tolstoy1989}
	{\sc V. N. Tolstoy},
	{\em Extremal projections for contragredient Lie algebras and superalgebras of finite growth},
	Russian Math. Surveys 44(1) (1989) 257--258.
	\href{https://doi.org/10.1070/RM1989v044n01ABEH002023}{doi:10.1070/RM1989v044n01ABEH002023}
\bibitem{Tolstoy2005} 
	{\sc V. N. Tolstoy},
	{\em Fortieth Anniversary of the Extremal Projector Method},
	in: Contemporary Mathematics 391, ``Non-commutative Geometry and Representation Theory in Mathematical Physics, Karlstad, July 2004'', 371--384, 2005.
\href{https://doi.org/10.1090/conm/391}{doi:10.1090/conm/391},
\href{https://doi.org/10.48550/arXiv.math-ph/0412087}{doi:10.48550/arXiv.math-ph/0412087}
\bibitem{Tolstoy2011} 
	{\sc V. N. Tolstoy},
	{\em Extremal Projectors for Contragredient Lie (Super)Symmetries (Short Review)},
	Physics of Atomic Nuclei, Vol. 74, No. 12, (2011) 1747--1757.
\bibitem{Zh1989}
	{\sc D. P. Zhelobenko}
	{\em Extremal projectors and generalized Mickelsson algebras over reductive Lie algebras},
	Mathematics of the USSR-Izvestiya, Vol. 33, No. 1, (1989), 85--100.
	\href{https://doi.org/10.1070/IM1989v033n01ABEH000815}{doi:10.1070/IM1989v033n01ABEH000815}
\bibitem{Zh1993}
	{\sc D. P. Zhelobenko},
	{\em Constructive modules and extremal projectors over Chevalley algebras},
	Funct. Anal. Appl. 27(3) (1993), 5--14.
\end{thebibliography}
\end{document}